\documentclass[12pt]{amsart}
\usepackage{accsupp}
\usepackage{amssymb,amscd,amsmath}
\usepackage{amsfonts}
\usepackage{mathtools}
\usepackage{fullpage}
\usepackage{array}
\usepackage{amsthm}
\usepackage{hyperref}
\usepackage{bm}
\usepackage{url}
\usepackage{mathrsfs}
\usepackage{todonotes}
\usepackage{relsize}
\usepackage{enumerate}
\usepackage{comment}
\numberwithin{equation}{section}
\newtheorem{theorem}{Theorem}[section]
\newtheorem{proposition}[theorem]{Proposition}
\newtheorem{lemma}[theorem]{Lemma}

\theoremstyle{definition}
\newtheorem{remark}[theorem]{Remark}

\newtheorem*{remark*}{Remark}
\newtheorem*{proposition*}{Proposition}
\newtheorem*{acknowledgement*}{Acknowledgement}
\newtheorem{example}[theorem]{Example}
\newtheorem{definition}[theorem]{Definition}
\newcommand\nnorm[1]{\left\lVert#1\right\rVert}
\newcommand\norm[1]{\left\lvert#1\right\rvert}
\newcommand{\overbar}[1]{\mkern 1.5mu\overline{\mkern-1.5mu#1\mkern-1.5mu}\mkern 1.5mu}

\newcommand{\U}{\mathrm{U}}

\newcommand{\C}{\mathbb{C}}
\newcommand{\F}{\mathbb{F}}
\newcommand{\Q}{\mathbb{Q}}
\newcommand{\Gal}{\mathrm{Gal}}
\newcommand{\Fqt}{\mathbb{F}_q[t]}
\newcommand{\Frob}{\mathrm{Frob}}

\newcommand*\from{\colon}
\usepackage{stmaryrd}

\begin{document}
\author{Leonhard Hochfilzer}
\date{\today}
\subjclass[2010]{primary: 11T55; secondary: 11M50, 11M38}
\keywords{$L$-functions, Galois representations}

\address{Mathematisches Institut, Bunsenstra{\ss}e 3-5, 37073 G\"ottingen, Germany.}
\email{leonhard.hochfilzer@mathematik.uni-goettingen.de}
\title{Variance of sums in short intervals and $L$-functions in $\mathbb{F}_q[t]$}
\maketitle

\begin{abstract}
Keating and Rudnick studied the variance of the polynomial von Mangoldt function $\Lambda \from \Fqt \to \C$ in arithmetic progressions and short intervals using two equidistribution results by Katz. Hall, Keating and Roditty-Gershon then generalised the result for arithmetic progressions for a von Mangoldt function $\Lambda_\rho$ attached to a Galois representation $\rho \from \Gal\left(\overbar{\F_q(t)}/\F_q(t)\right) \to \mathrm{GL}_m(\overbar{\Q}_\ell)$. We employ a recent equidistribution result by Sawin in order to generalise the corresponding result for short intervals for $\Lambda_\rho$.
\end{abstract}

\tableofcontents

\section{Introduction}

This paper is one of many recent studies of arithmetic quantities in short intervals in the function field setting \cite{bank2015, bank2018sums, keating2014variance, keating2016squarefree, keating2018sums, rodgers2018, rudnick2019angles}. Here we will be interested in the variance of a quantity to be defined below. The common strategy is to express the variance in terms of functions of coefficients of suitable $L$-functions and use an \emph{equidistribution} result in order to reduce these sums to a matrix integral, when taking the limit $q \rightarrow \infty$ as was done in \cite{keating2014variance, keating2016squarefree, keating2018sums, rodgers2018, rudnick2019angles}. In our case, the relevant equidistribution result is provided for us by Sawin \cite{sawin2018equidistribution}.

Let $\mathbb{F}_q$ be the finite field with $q$ elements, where $q = p^k$ for some positive integer $k$ and a rational prime $p$. We will assume the characteristic $p$ to be fixed throughout and when we write $q \rightarrow \infty$ we always mean $k \rightarrow \infty$.
In analogy with the integer case where we are interested in the distribution of prime numbers, one can ask questions about the distribution of irreducible polynomials in $\Fqt$ in short intervals and arithmetic progressions. For this we define the polynomial von Mangoldt function $\Lambda \from \Fqt \to \mathbb{Z}_{\geq 0}$  as
	\begin{equation}
		\Lambda(f) \coloneqq \begin{cases}
			\deg P \quad &\text{if } f = c \cdot P^k \text{ where } P \text{ is irreducible and } c \in \F_q^\times,  \\
			0 \quad &\text{otherwise},
		\end{cases}
	\end{equation}
and for a polynomial $A \in \Fqt$ of degree $n>h$, we define the interval $I(A;h)$ around $A$ of length $h$ to be
\begin{equation}
	I(A;h) \coloneqq \{ f \in \Fqt \mid \nnorm{f-A} \leq q^h \},
\end{equation}
where the usual norm is $\nnorm{f} \coloneqq q^{\deg(f)}$ for $f \neq 0$. Writing $\mathcal{M}_n \subset \Fqt$ for the monic polynomials of degree $n$ we thus define the two quantities
\begin{align}
\nu(A;h) &= \sum_{\substack{f \in I(A,h) \\ f(0) \neq 0}} \Lambda(f),	\quad \text{and} \\
\Psi(n;A,Q) &= \sum_{\substack{f \in \mathcal{M}_n \\ f \equiv A \; \mathrm{mod} \; Q}} \Lambda(f),
\end{align}
which essentially count the number of irreducible polynomials in short intervals and arithmetic progressions respectively. We define the expectation of $\nu(A;h)$ by
\begin{equation}
	\mathbb{E}_{A,n}[\nu(A;h)] = \frac{1}{q^n} \sum_{A \in \mathcal{M}_n}\nu(A;h),
\end{equation}
and the variance of $\nu(A;h)$ by 
\begin{equation}
	\mathrm{Var}_{A,n}[\nu(A;h)] = \frac{1}{q^n} \sum_{A \in \mathcal{M}_n} \norm{\nu(A;h) - \mathbb{E}_{B,n}[\nu(B;h)]}^2.
\end{equation}
The following two theorems were proved by Keating and Rudnick \cite[Theorem 2.1, Theorem 2.2]{keating2014variance}.
\begin{theorem}[Keating and Rudnick] \label{thm:keating_rudnick_short_intervals}
	Let $h$ and $n$ be non-negative integers with $h \leq n-5$. Then
	\begin{equation}
		\lim_{q \to \infty}\frac{1}{q^{h+1}} \mathrm{Var}_{A,n}[\nu(A;h)] = n-h-2.
	\end{equation}
\end{theorem}

\begin{theorem}[Keating and Rudnick] \label{thm:keating_rudnick_arith_prog}
	Let $Q$ be a square-free polynomial of positive degree and let $n$ be a positive integer such that $1 \leq \deg Q < n$. Then
	\begin{equation}
		\lim_{q \rightarrow \infty} \frac{1}{q^n}\sum_{\substack{A \; \mathrm{mod} \; Q \\ (A,Q) = 1}} \norm{\Psi(n;A,Q) - \frac{q^n}{\norm{\Phi(Q)}}}^2 = \deg(Q) -1.
	\end{equation}
\end{theorem}
These results were obtained using two equidistribution results, which Katz proved in \cite{katz2013square_free} and \cite{katz2013witt}, respectively. The two theorems above correspond to and provide further evidence for the well-known conjectures of Goldston-Montgomery \cite{goldston1987pair}  and Hooley \cite{hooley1974distribution} in the number field setting, which make predictions about the variance of rational primes in short intervals and arithmetic progressions, respectively. 

One can consider the same problem, but where we replace the polynomial von Mangoldt function $\Lambda$ with the von Mangoldt function $\Lambda_\rho$ attached to a Galois representation $\rho$ of the absolute Galois group of $K = \F_q(t)$. The following is a rough sketch of how $\Lambda_\rho$ is defined. The precise definition is given in Definition \ref{def.von_mangoldt_rho}.

Let $\mathcal{Z}(s) = \sum_{f \text{ monic}} \frac{1}{\nnorm{f}^s} = \frac{1}{1-q^{1-s}}$ be the zeta function of $\Fqt$. It is not hard to show the following identity
\begin{equation}
-\frac{\mathcal{Z}'(s)}{\mathcal{Z}(s)} = \sum_{\substack{f \in \Fqt \\ f \text{ monic}}}	\frac{\Lambda(f)}{\nnorm{f}^s}. 
\end{equation}
Thus given the Artin $L$-function $L_\rho(s)$ attached to $\rho$, we might naturally define the von Mangoldt function attached to $\rho$ by the coefficients of $-L_\rho'(s)/L_\rho(s)$. 

Once having defined the von Mangoldt function $\Lambda_\rho$ we can analogously to the above define
\begin{align*}
	\Psi_\rho(n;A,Q) &\coloneqq \sum_{\substack{f \equiv A \; \mathrm{mod} \; Q \\ f \in \mathcal{M}_n}}\Lambda_\rho(f), \; \text{and} \\
	\nu_\rho(A;h) &\coloneqq \sum_{\substack{f \in I(A;h) \\ f(0) \neq 0}}\Lambda_\rho(f)
\end{align*}
and the variances of these quantities.  In particular we write
\begin{align*}
	\mathrm{Var}_{A}[\Psi_\rho(n;Q,A)] &= \frac{1}{\norm{\Phi(Q)}} \sum_{\substack{A \; \mathrm{mod} \; Q \\ (A,Q) = 1}} \norm{\Psi(n;A,Q) - \mathbb{E}_{B}[\Psi(n;B,Q)]  }^2, \; \text{and} \\
	\mathrm{Var}_{A,n} \left[ \nu_\rho(A;h) \right] &= \frac{1}{q^n} \sum_{A \in \mathcal{M}_n} \norm{\nu_\rho(A;h) - \mathbb{E}_{B,n}[\nu_\rho(B;h)]}^2
\end{align*}

In this context, Hall, Keating and Roditty-Gershon proved the corresponding result for arithmetic progressions \cite{HKRG17}.

\begin{theorem}[Hall, Keating and Roditty-Gershon] \label{thm:arithmetic_progressions_galois}
	Let $\rho$ be a 'suitably nice' Galois representation with $q$-weight $w$ of $G_K$ depending on a square-free polynomial $Q \in \Fqt$ and write $\Phi(Q)$ for the set of Dirichlet characters of modulus $Q$. Then there exists a positive integer $r_{\mathcal{Q}}(\rho)$  depending on $Q$ and $\rho$ such that
	\begin{equation}
		\lim_{q \rightarrow \infty} \frac{\norm{\Phi(Q)}}{q^{n(1+w)}} \mathrm{Var}_{A}[\Psi_\rho(n;Q,A)] = \mathrm{min}\{n,r_\mathcal{Q}(\rho) \}.
	\end{equation}
\end{theorem} 
The main result of this manuscript is the corresponding asymptotic for short intervals using a recent equidistribution result by Sawin \cite{sawin2018equidistribution}. Using this, in the end one obtains the following matrix integral
\begin{equation*}
	\int_{\U(S)} \norm{\mathrm{Tr}(g^n)}^2 dg,
\end{equation*}
which can be shown to equal $\mathrm{min}\{ n,S \}$ \cite[Theorem 2.1]{diaconis2001linear}.
\begin{theorem} \label{thm:main_result_intro_form}
	Let $\rho$ be a 'suitably nice' Galois representation of $G_K$ of $q$-weight $w$ depending on $Q = t^{n-h}$ where $h$ is a nonnegative integer such that $h \leq n-5$. Then there exists a positive integer $s_{\mathcal{Q}}(\rho)$ depending on $Q$ and $\rho$ such that
	\begin{equation}
		\lim_{q \to \infty} \frac{1}{q^{nw+h+1}} \mathrm{Var}_{A,n} \left[ \nu_\rho(A;h) \right] = \mathrm{min}\{n,s_\mathcal{Q}(\rho) \}.
	\end{equation}
\end{theorem}
Later we will make precise what we mean by 'suitably nice', we will define what it means for $\rho$ to have $q$-weight $w$ and we will define the integer $s_{\mathcal{Q}}(\rho)$ once the necessary notions have been introduced. For reference, the above theorem is stated in its complete form and with all assumptions in Theorem~\ref{thm:main_result_full_form}, respectively.
\begin{example}
By taking $\rho = \bm{1}$ it is possible to deduce Theorem~\ref{thm:keating_rudnick_short_intervals} and Theorem~\ref{thm:keating_rudnick_arith_prog} from Theorem~\ref{thm:main_result_intro_form} and Theorem~\ref{thm:arithmetic_progressions_galois}, respectively. In fact,  then  $w=0$ and  one can check using \cite[Lemma 4.10]{sawin2018equidistribution}  $r_{\mathcal{Q}}(\rho) = \deg Q -1$ and $s_{\mathcal{Q}}(\rho) = n-h-2$.
\end{example}

\begin{example} \label{ex.legendre_curve}
	Let $E$ be the Legendre curve 
	\begin{equation*}
		E \colon \; y^2 = x(x-1)(x-t)
	\end{equation*}
	over $\mathbb{F}_{q_0}(t)$ where the characteristic of $\mathbb{F}_{q_0}$ exceeds $3$. Its associated Tate module defines a tamely ramified rank $2$ Galois representation $\rho_E$. Then Lemma 4.10 and Example 4.11 in \cite{sawin2018equidistribution} show 
	\begin{equation*}
		s_{\mathcal{Q}}(\rho_E) = 2(n-h-1).
	\end{equation*}
	It can be shown that $\rho$ satisfies the conditions of Theorem~\ref{thm:main_result_intro_form}, and in this case $\rho$ has weight $w = 1$. Thus  for the variance of the associated von Mangoldt function $\Lambda_{\rho_E}$ we obtain
	\begin{equation*}
		\lim_{q \to \infty} \frac{1}{q^{n+h+1}} \mathrm{Var}_{A,n} \left[ \nu_{\rho_E}(A;h) \right] = \mathrm{min}\{n, 2(n-h-1) \}.
	\end{equation*}
\end{example}

\subsection*{Comparison with number field setting}
Let $F$ be a primitive $L$-function in the Selberg class of degree $d_F$. One can define the von-Mangoldt function $\Lambda_F$ attached to $F$ via the following equation
\begin{equation*}
	-\frac{F'(s)}{F(s)} = \sum_{n=1}^\infty \frac{\Lambda_F(n)}{n^s}.
\end{equation*}
Writing
\begin{equation*}
	\psi_F(x) = \sum_{n \leq x} \Lambda_F(x),
\end{equation*}
we expect a general prime number theorem to hold
\begin{equation*}
	\psi_F(x) = m_F x + o(x),
\end{equation*}
where $m_F$ is the order of the pole of $F(s)$ at $s=1$. We study the variance of the von Mangoldt function $\Lambda_F$ by considering
\begin{equation} \label{eq.variance_general_L}
	\tilde{V}_F(X,h) = \int_1^X \norm{\psi_F(x+h) - \psi_F(x) - m_Fh}^2dx.
\end{equation}
If one takes $F(s) = \zeta(s)$, the Riemann Zeta function, then the expression $\tilde{V}_F(X,h)$ measures the variance of the number of rational primes in short intervals.  For general $F$ the quantity defined in \eqref{eq.variance_general_L} has been studied in \cite{bui2016variance}. 
There it is shown under the Generalised Riemann Hypothesis (GRH) that the analogous results in the number field setting are equivalent to extensions of pair correlation conjectures regarding the zeroes of $F(s)$ on the critical line. In particular, assuming GRH and these pair correlation conjectures the following was shown \cite[Theorem C1, Theorem C2]{bui2016variance}. We write $\mathfrak{q}_F$ for the conductor of $F(s)$ and $\gamma_0$ for the Euler-Mascheroni constant.

Let $\varepsilon > 0$.  If $0<B_1 < B_2 \leq B_3 < 1/d_F$ then there exists some $c>0$ such that
\begin{multline} \label{eq.number_field_degree_1}
	\tilde{V}_F(X,h) = hX \left(  d_F \log \frac{X}{h} + \log \mathfrak{q}_F - (\gamma_0 + \log 2 \pi)d_F \right) \\
	+ O_\varepsilon \left( hX^{1+\varepsilon}(h/X)^{c/3} \right) + O_\varepsilon \left( hX^{1+\varepsilon} \left( hX^{B_1-1} \right)^{1/3(1-B_1)} \right)
\end{multline}
uniformly for $X^{1-B_3} \ll h \ll X^{1-B_2}$. 

If $1/d_F<B_1 < B_2 \leq B_3 < 1$ then there exists some $c > 0$ such that
\begin{multline} \label{eq.number_field_degree_large}
	\tilde{V}_F(X,h) = \frac{1}{6} hX \left(  6\log X   - (3 + 8 \log 2) \right) \\
	+ O_\varepsilon \left( hX^{1+\varepsilon}(h/X)^{c/3} \right) + O_\varepsilon \left( hX^{1+\varepsilon} \left( hX^{B_1-1} \right)^{1/3(1-B_1)} \right)
\end{multline}
uniformly for $X^{1-B_3} \ll h \ll X^{1-B_2}$.

If $d_F = 1$ then the second condition is never satisfied so $\tilde{V}_F(X,h)$ is always given by \eqref{eq.number_field_degree_1}. However, if $d_F \geq 2$ there are two different regimes governing the growth of $\tilde{V}_F(X,h)$ depending on the length of the interval $h$. Note that the asymptotic behaviour of \eqref{eq.number_field_degree_1} and \eqref{eq.number_field_degree_large} are qualitatively different; in the range governed by \eqref{eq.number_field_degree_1} the leading term of $\tilde{V}_F(X,h)/Xh$ is proportional to $\log h$, while in the range governed by \eqref{eq.number_field_degree_large} the leading term of $\tilde{V}_F(X,h)/Xh$ is independent of $h$

This corresponds to what we obtain in Theorem \ref{thm:main_result_intro_form}; if the degree of the considered $L$-function exceeds $1$ then the leading order coefficient is governed by two different regimes depending on the length of the interval. Example~\ref{ex.legendre_curve} provides an explicit case in which this can be clearly seen.

\subsection*{Acknowledgements}
The author would like to thank Jon Keating for suggesting this area of research and for many useful comments. Further the author is grateful to Chris Hall for providing help in Section \ref{section.weights_purity} and pointing out some relevant results. Finally the author thanks Theo Assiotis, Ofir Gorodetsky, Edva Roditty-Gershon, Zeev Rudnick, Will Sawin and Damaris Schindler for helpful conversations and comments.

\section{Preliminaries} \label{sec.preliminaries}
Fix a prime $p$ and a prime power $q_0 = p^r$ for some $r$.  When we write $q \rightarrow \infty$ we  always mean $q = q_0^k$ and $k \rightarrow \infty$.
 Consider the field $K = \mathbb{F}_q(t)$ of rational functions with coefficients in $\mathbb{F}_q$. Fix an algebraic closure $\overbar{K}$ of $K$ and a separable closure $K_s$ of $K$. 
 Denote by $G_K = \Gal(K_s/K)$ the \emph{absolute Galois group of} $K$.

Let $\mathcal{P}$ be the set of places of $K$.  If we 
 and let $\mathcal{I}$ be the monic irreducible polynomials in $\Fqt$. By Ostrowski's Theorem we have a correspondence correspondence 
\begin{align*}
     \mathcal{I} \cup \{ \infty \} &\longleftrightarrow  \mathcal{P} \quad \quad  \\
	P &\longrightarrow v_P \\
P_v &\longleftarrow v
\end{align*}

Given a subset of places $\mathcal{Q} \subset \mathcal{P}$ we denote by $K_{\mathcal{Q}} \subseteq K_s$ the maximal subextension subject to all places in $\mathcal{P} \setminus \mathcal{Q}$ being unramified. We denote the Galois group by
\begin{equation}
	G_{K,\mathcal{Q}} \coloneqq \Gal(K_{\mathcal{Q}}/K)
\end{equation}
Given a place $v$ we define as usual 
	\begin{equation*}
		\mathcal{O}_v = \{ x \in K \mid v(x) \leq 1 \},
	\end{equation*} 
	which is a ring with maximal ideal
	\begin{equation*}
		\mathfrak{m}_v = \{ x \in K \mid v(x) < 1\}.
	\end{equation*}
	We define the residue class field to be
\begin{equation}
	\kappa_v \coloneqq \mathcal{O}_v/\mathfrak{m}_v,
\end{equation}
and the \emph{degree}  of $v$ is defined to be 
\begin{equation}
	d_v \coloneqq [\kappa_v : \F_q].
\end{equation}
Let $v$ be a place on $K = \F_q(t)$, and fix an extension of places $w\vert v$ on $K_s$. By definition of the decomposition group $D_w$, for any $\sigma \in D_w$ we have $\sigma(\mathcal{O}_w) = \mathcal{O}_w$ and $\sigma(\mathfrak{m}_w) = \mathfrak{m}_w$. Further, since $\sigma \vert_K = \mathrm{id}$ we also have $\sigma \vert_{\kappa_v} = \mathrm{id}$, so that any $\sigma \in D_w$ induces a $\kappa_v$-automorphism $\overbar{\sigma}$ as follows
\begin{equation}
	\overbar{\sigma} \from \kappa_w \to \kappa_w, \quad x \; \; \mathrm{mod} \; \mathfrak{m}_w \mapsto \sigma(x)  \; \mathrm{mod} \; \mathfrak{m}_w.
\end{equation}
Therefore we obtain a surjective homomorphism
\begin{equation}
	D_w \to \Gal(\kappa_w/\kappa_v),
\end{equation}
with kernel being the inertia group $I_w$. If we write $G_w = D_w/I_w$ we thus have
\begin{equation}
	G_w \cong \Gal(\kappa_w/ \kappa_v).
\end{equation}
From a computation above, we know that $\kappa_v \cong \F_{q^{d_v}}$, and since $K_s$ is the separable closure it is not hard to see that we must have $\kappa_w \cong \overline{\F}_{q^{d_v}}$. As $\overline{\F}_{q^{d_v}}$ is the separable closure of $\F_{q^{d_v}}$, 
we therefore get
\begin{equation}
	G_w \cong \Gal(\overline{\F}_{q^{d_v}}/ \F_{q^{d_v}}).
\end{equation}
Consider $\tau \in \Gal \left(\overline{\F}_{q^{d_v}}/ \F_{q^{d_v}}\right)$ given by $\tau = \left(a \mapsto (a)^{q^{d_v}}\right)$, for $a \in \overline{\F}_{q^{d_v}}$.
By the isomorphism above, there exists an element in $G_w$ corresponding to $\tau$, the so-called \emph{Frobenius element}  which we denote by $\Frob_w \in G_w$. Clearly there is a surjective map $D_w \twoheadrightarrow G_w$ with Kernel $I_w$, and we denote the preimage of $\Frob_w$ under this map by $\Frob_w(I_w)$.

If we had two different places $w, w'$ of $K_s$ extending  $v$ then it is not hard to see that there exists $\sigma \in G_K$ such that simultaneously
\begin{equation}
	\sigma^{-1}\Frob_w \sigma = \Frob_{w'}, \quad \sigma^{-1} G_w \sigma = G_{w'}, \quad \text{and} \quad \sigma^{-1} I_w \sigma = I_{w'}.
\end{equation}
In the following we will only be interested in the determinant or the trace of the Frobenius automorphism, and so any extension of $v$ works equally well. 
Therefore, since $\Frob_w$ is determined by $v$ up to conjugacy we will abuse notation and write $v=w$. \section{Artin $L$-functions} \label{sec:artin_l}
In this section we construct that $L$-function attached to a Galois representation of $\Gal(K_s/K)$. This will be our central tool when defining $\Lambda_\rho$. 
\subsection{$L$-functions}
Let $\ell$ be a prime different to $p$, and consider an algebraic closure $\overbar{\Q}_\ell$ of $\Q_\ell$. Let $V$ be a finite dimensional $\overbar{\Q}_\ell$ vector space, and let $\mathcal{S} \subset \mathcal{P}$ be a finite subset of places of $K$. Consider a continuous homomorphism
\begin{equation}
	\rho \from G_{K,\mathcal{S}} \to \mathrm{GL}(V).
\end{equation}
Now fix a place $v \in \mathcal{P}$, and define
\begin{equation}
	V_v \coloneqq \{ \mathbf{x} \in V \mid \rho(\sigma)(\mathbf{x}) = \mathbf{x}, \text{ for all } \sigma \in I_v \}.
\end{equation}
By construction, the inertia subgroup $I_v$ acts trivially on $V_v$, and the action by the decomposition group $D_v$ preserves $V_v$. 
Therefore $\rho$ induces a representation $\rho_v$ on the quotient $G_v = D_v/I_v$,
\begin{equation}
	\rho_v \from G_v \to \mathrm{GL}(V_v).
\end{equation}
We will now define the $L$-function \emph{attached}  to the representation $\rho$. 

\begin{definition}
	Let $\mathcal{Q} \subset \mathcal{P}$ be a finite subset of places of $K$. 
	Define the \emph{local Euler factor}   of $\rho$ at a place $v$ to be
	 the inverse characteristic polynomial of $\rho(\Frob_v) \in V_v$, that is
\begin{equation} \label{local_euler_factor_def}
	L(T,\rho_v) \coloneqq \det(I - T \rho_v(\Frob_v)\mid V_v) \in \overbar{\Q}_\ell[T].
\end{equation} 
	We define the \emph{partial (Artin) $L$-function attached to $\rho$}   by
	\begin{equation} \label{partial_l_defn}
		L_{\mathcal{Q}}(T,\rho) \coloneqq \prod_{v \notin \mathcal{Q}}L(T^{d_v},\rho_v)^{-1} \in \overbar{\Q}_\ell \llbracket  T \rrbracket,
	\end{equation} 
	and we define the \emph{complete (Artin) $L$-function attached to $\rho$}   by
	\begin{equation} \label{complete_l_defn}
		L(T,\rho) \coloneqq \prod_{v \in  \mathcal{P}}L(T^{d_v},\rho_v)^{-1} \in \overbar{\Q}_\ell \llbracket  T \rrbracket,
	\end{equation}
	where $\overbar{\Q}_\ell \llbracket  T \rrbracket$ is the ring of formal power series with coefficients in $\overbar{\Q}_\ell$.
\end{definition}

Using the theory of middle extension sheaves, one can deduce from Deligne's theorem that both the partial and the complete $L$-functions of $\rho$ are rational functions \cite[3.4]{HKRG17}. Therefore it makes sense to talk about $\deg(L_\mathcal{Q}(T,\rho))$ and $\deg(L(T,\rho))$.

\subsection{Trace formula} \label{subsec:trace_formulas}

We now want to compute the logarithmic derivative of $L(T,\rho)$ in order to define a von Mangoldt function $\Lambda_\rho$. Note that by definition of the local Euler factor \eqref{local_euler_factor_def}, we have 
\begin{equation}
	L(T,\rho_v) = (1-\lambda_1T) \cdots (1-\lambda_{d_v}T),
\end{equation}
where $\lambda_i \in \overbar{\Q}_\ell$ are the eigenvalues of $\rho_v(\Frob_v)$. Therefore computing the logarithmic derivative of $L(T,\rho_v)$ we get
\begin{equation}
	T \frac{d}{dT}\log L(T,\rho_v)^{-1} = T\sum_{k=1}^{d_v}\lambda_k \left(\frac{\prod_{i\neq k}(1-\lambda_i T)}{\prod_{j=1}^{d_v}(1-\lambda_j T)}\right) 
	= T\sum_{k=1}^{d_v} \lambda_k (1-\lambda_kT)^{-1}.
\end{equation}
From the above computation we obtain
\begin{equation}
	T \frac{d}{dT}\log L(T,\rho_v)^{-1} = \sum_{m=1}^{\infty}\left( \sum_{k=1}^{d_v}\lambda_k^m \right)T^m.
\end{equation}
Therefore, defining
\begin{equation}\label{Lefschetz_trace}
	a_{\rho,v,m} \coloneqq \mathrm{Tr}(\rho_v(\Frob_v)^m \mid V_v),
\end{equation}
we can rewrite this as 
\begin{equation} \label{local_l_derivative}
	T \frac{d}{dT}\log L(T,\rho_v)^{-1} = \sum_{m=1}^{\infty}a_{\rho,v,m}T^m.
\end{equation}
Further, if we let $\mathcal{P}_d \subset \mathcal{P}$ be the places of degree $d$, and define
\begin{equation}
	b_{\rho,n} \coloneqq \sum_{md = n}\sum_{v \in \mathcal{P}_d \setminus \mathcal{Q}} d \cdot a_{\rho,v,m},
\end{equation}
 then \eqref{local_euler_factor_def} together with \eqref{partial_l_defn} implies
\begin{equation}
	T \frac{d}{dT}\log L_{\mathcal{Q}}(T,\rho) = \sum_{n=1}^{\infty}b_{\rho,n} T^n.
\end{equation}
The quantities $b_{\rho,n}$ have arithmetic meaning --- they are the so-called \emph{cohomological traces} of $\rho$, and our definition \eqref{Lefschetz_trace} is usually a result known as the \emph{Grothendieck-Lefschetz trace formula}. Later we will make crucial use of this arithmetic origin. 
\section{The von Mangoldt function $\Lambda_\rho$} \label{sec:defn_of_von_mang}
Let $\mathcal{M} \subset \Fqt$ be the set of monic polynomials, and let $\mathcal{M}_n \subset \mathcal{M}$ be the set of monic polynomials of degree $n$. Further, let $\mathcal{I} \subset \mathcal{M}$ be the set of irreducible monic polynomials. 
Recall we write $v_P$ for the finite place corresponding to $P \in \mathcal{I}$.
\begin{definition} \label{def.von_mangoldt_rho}
	Let $\rho \from G_{K,\mathcal{Q}} \to \mathrm{GL}(V)$ a continuous $\ell$-adic Galois representation, and let $a_{\rho,v,m} = \mathrm{Tr}\left( \rho_v(\Frob_v) \right)$. We define the \emph{von Mangoldt function} $\Lambda_p \from \Fqt \to \overbar{\Q}_\ell$ by
	\begin{equation} \label{eq:von_mang_defn}
		\Lambda_\rho(f) = \begin{cases} d \cdot a_{\rho,v_P,m} &\text{ if } f = c \cdot P^m \text{ for some } c \in \F_q^{\times} \text{ and } P \in \mathcal{I}_d, \\
		0 &\text{ otherwise}.
		\end{cases} 
	\end{equation}
\end{definition}
If we compare this with the usual polynomial von Mangoldt function, recall that the polynomial zeta function is given by
\begin{equation}
	Z(T) = \sum_{f \in \mathcal{M}_n} T^n = \frac{1}{1-qT}.
\end{equation}
Defining local Euler factors $Z(T,P) = (1-T)$, then  just like in \eqref{complete_l_defn} we obtain
\begin{equation}
	Z(T) = \prod_{P \in \mathcal{I}}Z(T^{\deg(P)},P)^{-1}.
\end{equation}
An easy computation shows
\begin{equation}
	T\frac{d}{dT} \log Z(T,P)^{-1} = \sum_{m=1}^\infty T^m,
\end{equation}
and thus
\begin{equation}
	T\frac{d}{dT} \log Z(T^{\deg(P)},P)^{-1} = \sum_{m=1}^\infty \deg(P) T^{\deg(P) m} = \sum_{n=1}^\infty\left( \sum_{m \deg(P) = n} \Lambda(P^m) \right) T^n,
\end{equation}
since $\Lambda(P^m) = \deg(P)$ for all $P \in \mathcal{I}$ and $m \geq 1$. This also shows that for the trivial representation $\rho = \mathbf{1}$ we recover $\Lambda_{\mathbf{1}} = \Lambda$. It is also worth noting that in general, for a representation $\rho$, the von Mangoldt function $\Lambda_\rho$ does \emph{not}  satisfy $\Lambda_\rho(P^m) = \Lambda_{\rho}(P^k)$ for positive integers $m$ and $k$.

A priori $\Lambda_\rho$ takes values in $\overbar{\Q}_\ell$, however we will assume that the range lies within $\overbar{\Q}$, as was done in \cite{HKRG17}. Further, we fix embeddings $\iota \from \overbar{\Q} \to \C$, and $\overbar{\Q} \to \overbar{\Q}_\ell$ and later we will impose appropriate assumptions on $\iota$. We note here that we suppress notation and mostly not make the embedding $\iota$ explicit in the hope of improving readability.

Let $A \in \Fqt$ be a polynomial of degree $n$. 
In analogy with the polynomial case, we define the following quantity
\begin{equation}
	\nu_\rho(A;h) \coloneqq \sum_{\substack{f \in I(A;h) \\ f(0) \neq 0}}\Lambda_\rho(f),
\end{equation}
and also we define the expectation and variance of $\nu_\rho(A;h)$ in the usual way
\begin{align} \label{variance_mangoldt_short_intervals_definition}
	\mathbb{E}_{A,n}[\nu_\rho(A;h)] &\coloneqq \frac{1}{q^n} \sum_{A \in \mathcal{M}_n}\nu_\rho(A;h), \quad \text{ and} \\
	\mathrm{Var}_{A,n}[\nu_\rho(A;h)] &\coloneqq \frac{1}{q^n} \sum_{A \in \mathcal{M}_n} \lvert \nu_\rho(A;h) - \mathbb{E}_{A,n}[\nu_\rho(A;h)] \rvert^2,
\end{align}
respectively. Our aim is to establish an asymptotic for the variance. In order to do this we need to introduce some more ideas. In particular it will be very useful to express the variance as a combination of Dirichlet characters, which naturally have modulus $T^{n-h}$.

\section{Dirichlet characters} 
Let $Q \in \Fqt$ be a non-constant polynomial and write $\Gamma(Q) = \left( \Fqt/Q\Fqt\right)^\times$. 
Recall that $\mathcal{P}$ is the set of places of $K$, and let $\mathcal{Q}$ be the set 
\begin{equation}
	\mathcal{Q} \coloneqq \{ v \in \mathcal{P} \mid v(Q) \neq 1 \}.
\end{equation}
Note that for any irreducible polynomial $P \in \mathcal{I}$ we have
\begin{equation}
	P \vert Q \iff v_P \in \mathcal{Q}.
\end{equation}
Further, since we assumed $Q$ not to be constant, we have that $v_\infty(Q) \neq 1$, so that $v_\infty \in \mathcal{Q}$. Consider now the complement $\mathcal{U_Q} \coloneqq \mathcal{P} \setminus \mathcal{Q}$. Rephrasing the above, we have a bijective correspondence between the elements $u \in \mathcal{U_Q}$ and monic irreducible polynomials $P_u$, which do not divide $Q$. 

Recall that by definition $G_{K,\mathcal{Q}} = \Gal(K_{\mathcal{Q}}/K)$, where $K_\mathcal{Q} \subset K_s$ is the maximal extension subject to all places in $\mathcal{P} \setminus \mathcal{Q} = \mathcal{U_Q}$ being unramified. As was outlined in Section~\ref{sec.preliminaries}, every $u \in \mathcal{U_Q}$ corresponding to $P_u$  therefore determines a conjugacy class $(\Frob_u) \in G_{K,\mathcal{Q}}$. In particular, if we consider the abelianization $G_{K,\mathcal{Q}}^{\mathrm{ab}}$ of $G_{K,\mathcal{Q}}$, then we obtain a well-defined element $\Frob_u \in G_{K,\mathcal{Q}}^{\mathrm{ab}}$, and a surjective homomorphism 
\begin{align}
	\alpha_\mathcal{Q} \from G_{K,\mathcal{Q}}^{\mathrm{ab}} &\twoheadrightarrow \Gamma(Q) \\
	\Frob_u &\mapsto P_u \; \mathrm{mod} \; Q.
\end{align}
Clearly, there is also a canonical map $G_{K,\mathcal{Q}} \twoheadrightarrow G_{K,\mathcal{Q}}^{\mathrm{ab}}$ and so a character $\varphi \in \Phi(Q)$ induces a map
\begin{equation} \label{lift_character}
	G_{K,\mathcal{Q}} \twoheadrightarrow G_{K,\mathcal{Q}}^\mathrm{ab} \overset{\alpha_{\mathcal{Q}}}\twoheadrightarrow \Gamma(Q) \overset{\varphi}\rightarrow \overbar{\Q}^\times.
\end{equation}
Let $\Phi(Q)$ be the set of Dirichlet characters of modulus $Q$. We may view a character $\varphi \in \Phi(Q)$ as a map $\varphi \from \Fqt \to \overbar{\mathbb{Q}}$ in the usual way. This allows us to lift $\varphi$ to a map $G_{K,\mathcal{Q}} \to \overbar{\Q}$ in a similar fashion to above. That is, we define it to be the composite
\begin{equation}
	G_{K,\mathcal{Q}} \twoheadrightarrow G_{K,\mathcal{Q}}^\mathrm{ab} \overset{\alpha_{\mathcal{Q}}}\twoheadrightarrow \Fqt \overset{\varphi}\rightarrow \overbar{\Q}.
\end{equation}
We will abuse notation and denote the induced map $G_{K,\mathcal{Q}} \to \overbar{\Q}$ by $\varphi$.

Recall, a character $\varphi \in \Phi(Q)$ is \emph{even} if $\varphi(a) = 1$ for all $a \in \F_q^\times$ and denote by $\Phi(Q)^{ev} \subset \Phi(Q)$ the subset of even Dirichlet characters of modulus $Q$. Let $c$ be a generator of $\F_q^\times$. An arbitrary Dirichlet character $\varphi \in \Phi(Q)$ can take the value $\varphi(c) = \zeta^i$ for $i = 0, \dots, q-1$, where $\zeta$ is a primitive $(q-1)$-th root of unity, whereas even characters are restricted to satisfy $\varphi(c) = 1$. From this we get the following.
\begin{lemma} \label{lem:even_characters}
Consider $Q \in \Fqt$. Then
\begin{equation}
		\norm{\Phi(Q)^{ev}} = \frac{\norm{\Phi(Q)}}{q-1}.
	\end{equation}	
\end{lemma}

\section{Twisted $L$-functions}
\subsection{$L$-functions}

As before, for $\mathcal{S} \subset \mathcal{P}$ finite, let 
\begin{equation}
	\rho \from G_{K,\mathcal{S}} \to \mathrm{GL}(V)
\end{equation}
be a continuous $\ell$-adic Galois representation. Let $\mathcal{Q} \subset \mathcal{P}$ be a finite subset of places. Consider now any group homomorphism
\begin{equation}
	\varphi \from G_{K,\mathcal{Q}} \to \overbar{\Q}_\ell,
\end{equation}
which we will call an \emph{$\ell$-adic character, with conductor supported in $\mathcal{Q}$}. For our purposes $\varphi$ will usually be induced from a Dirichlet character.

Consider $\mathcal{R} = \mathcal{S} \cup \mathcal{Q}$. Recall, that the definition of the fields $K_\mathcal{S}$ and $K_\mathcal{Q}$ was that they are the maximal subextensions of $K_s/K$, unramified away from $\mathcal{S}$ and $\mathcal{Q}$, respectively. Since $\mathcal{R} \supset \mathcal{S,Q}$, fewer  places are subject to being unramified for the extension $K_\mathcal{R}$ compared to $K_\mathcal{S}$ and $K_\mathcal{Q}$. Therefore $K_\mathcal{R} \supset K_\mathcal{S},K_\mathcal{Q}$, whence by Galois theory, there exist canonical, surjective quotient maps
\begin{equation}
	G_{K,\mathcal{R}} \twoheadrightarrow G_{K,\mathcal{Q}}, \quad \text{and} \quad G_{K,\mathcal{R}} \twoheadrightarrow G_{K,\mathcal{S}}.
\end{equation}
Therefore we can canonically define new maps $\rho_\mathcal{R}$ and $\varphi_\mathcal{R}$ as the composites
\begin{equation}
	\rho_\mathcal{R} \from G_{K,\mathcal{R}} \twoheadrightarrow G_{K,\mathcal{S}} \overset{\rho}\rightarrow \mathrm{GL}(V), \quad \text{and} \quad \varphi_\mathcal{R} \from G_{K,\mathcal{R}} \twoheadrightarrow G_{K,\mathcal{Q}} \overset{\varphi}\rightarrow \overbar{\Q}_\ell.
\end{equation}
We define the \textit{tensor product} $\rho \otimes \varphi$ to be the continuous homomorphism, given by 
\begin{align}
	\rho \otimes \varphi \from G_{K,\mathcal{R}} &\to \mathrm{GL}(V) \quad \qquad \\
	g &\mapsto  \rho_\mathcal{R}(g)\varphi_\mathcal{R}(g).
\end{align}
This allows us to define twisted $L$-functions.

\begin{definition} \label{def:twisted_l_functions}
	Let $\rho$ and $\varphi$ be as above. The \emph{Euler factors attached to $\rho$, twisted  by $\varphi$ at a place $v$}  are defined to be
	\begin{equation}
		L(T,(\rho \otimes \varphi)_v) \coloneqq \det(I-T(\rho \otimes \varphi)_v(\Frob_v) \mid V_v).
	\end{equation}
	The \emph{partial}  and \emph{complete}  $L$-functions \emph{attached to}  $\rho$, \emph{twisted by}   $\varphi$ are respectively defined to be
	\begin{equation}
		L_\mathcal{Q}(T,\rho \otimes \varphi) \coloneqq \prod_{v \notin \mathcal{Q}} L(T^{d_v}, (\rho \otimes \varphi)_v)^{-1} \in \overbar{\Q}_\ell \llbracket T \rrbracket ,
	\end{equation}
	and 
	\begin{equation}
		L(T,\rho \otimes \varphi) \coloneqq \prod_{v \in \mathcal{P}} L(T^{d_v}, (\rho \otimes \varphi)_v)^{-1} \in \overbar{\Q}_\ell \llbracket T \rrbracket.
	\end{equation}
\end{definition}

\begin{remark}
Here we mean 
\begin{equation}
	V_v \coloneqq \{ \mathbf{x} \in V \mid (\rho \otimes \varphi) (\sigma)(\mathbf{x}) = \mathbf{x}, \text{ for all } \sigma \in I_v \}.
\end{equation}	
Further, for any $v \notin \mathcal{Q}$, we have $(\rho \otimes \varphi)_v = \rho_v \cdot \varphi$, so that 
\begin{equation} \label{eq:euler_factor_twist_nice}
	L(T,(\rho \otimes \varphi)_v) = L(\varphi(\Frob_v)T,\rho_v).
\end{equation}
\end{remark}

\subsection{Trace formula for twists}
Let $v \notin \mathcal{Q}$. Using the results from Section~\ref{subsec:trace_formulas} and \eqref{eq:euler_factor_twist_nice} we can find expressions for the logarithmic derivatives of the twisted Euler factors. Indeed, we obtain
\begin{equation}
T \frac{d}{dT} \log L(T,(\rho \otimes \varphi)_v	)^{-1} = \sum_{m=1}^\infty \varphi(\Frob_v)^m a_{\rho,v,m}T^m.
\end{equation}
Let $\mathcal{P}_d \subset \mathcal{P}$ the places of degree $d$. If we define  
\begin{equation} \label{eq:lefshetz_grothendieck_twisted}
	b_{\rho \otimes \varphi,n} \coloneqq \sum_{md = n} \sum_{v \in \mathcal{P}_d\setminus \mathcal{Q}} d \cdot \varphi(\Frob_v)^m a_{\rho,v,m},
\end{equation}
then we get
\begin{equation} \label{eq:twisted_L-function_in_terms_of_coh_traces}
	T \frac{d}{dT} \log L_\mathcal{Q}(T,\rho \otimes \varphi	) = \sum_{n=1}^\infty b_{\rho \otimes \varphi,n} T^n.
\end{equation}
Regarding a character $\varphi$ as a map $\varphi \from G_{K,\mathcal{Q}} \to \overbar{\Q}$, as before, note that  for an irreducible polynomial $P \in \Fqt$ by definition we have $\varphi(P) = \varphi(\Frob_{v_P})$. Thus, from \eqref{eq:lefshetz_grothendieck_twisted} and the definition of $\Lambda_\rho$ \eqref{eq:von_mang_defn} we deduce
\begin{equation} \label{eq:cohomological_traces_character_sum}
	b_{\rho \otimes \varphi,n} = \sum_{f \in \mathcal{M}_n} \varphi(f) \Lambda_\rho(f).
\end{equation}

\section{The variance as a sum of characters}
\subsection{Relating  short intervals to arithmetic progressions}
We define sums of $\Lambda_\rho$ in arithmetic progressions by
\begin{equation}
	\Psi_\rho(n;A,Q) \coloneqq \sum_{\substack{f \equiv A \; \mathrm{mod} \; Q \\ f \in \mathcal{M}_n}}\Lambda_\rho(f).
\end{equation}
We also define a related quantity $\widetilde{\Psi}_\rho(n;A,Q)$, where we consider sums of $\Lambda_\rho$ in arithmetic progressions, however, instead of summing over monic polynomials we now also admit non-monic ones. That is, we define
\begin{equation}
	\widetilde{\Psi}_\rho(n;A,Q) \coloneqq \sum_{\substack{f \equiv A \; \mathrm{mod} \; Q \\ \deg(f) = n}}\Lambda_\rho(f).
\end{equation}
Our aim is to relate $\nu_\rho(A;h)$ and $\widetilde{\Psi}_\rho(n;A,Q)$ following the approach by Keating and Rudnick in \cite{keating2014variance}. 
In order to do this, we will first need to define the \emph{involution}  of a polynomial.

\begin{definition}
	Let $f(t) = a_dt^d + \dots + a_0 \in \Fqt$ be a non-zero polynomial of degree $d$. We define the \emph{involution}  $f^*(t)$ of $f(t)$ to be
	\begin{equation}
		f^*(t) \coloneqq t^d f\left(\frac{1}{t}\right) = a_d + a_{d-1}t + \dots + a_0t^d.
	\end{equation}
	Further, we define $0^* = 0$.
\end{definition}
We record the following properties of the involution, which follow directly from the definition:
\begin{itemize}
	\item[(i)] $f^*(0) \neq 0$ and $f(0) \neq 0$ if and only if $\deg f = \deg f^*$,
	\item[(ii)] If $f(0) \neq  0$ then $(f^*)^* = f$, 
	\item[(iii)] $(fg)^* = f^*g^*$, and
	\item[(iv)] $f$ is monic if and only if $f^*(0) = 1$.
\end{itemize}

The proof of the next lemma, follows the proof of \cite{keating2014variance}, where the corresponding result for the case $\rho = \mathbf{1}$ is shown.

\begin{lemma}
	Let $f \in \Fqt$ be a polynomial such that $f(0) \neq 0$. Then
	\begin{equation}
		\Lambda_\rho(f) = \Lambda_\rho(f^*).
	\end{equation}
\end{lemma}
\begin{proof}
	First we want to establish the following fact: Assuming $f(0) \neq 0$, we have that $f$ is irreducible if and only if $f^*$ is irreducible. Indeed, if we can factorise $f = ab$, where $a \in \Fqt$ and $b \in \Fqt$ are polynomials of positive degree, then by multiplicativity of the involution we have $f^* = a^*b^*$. Since $f(0) \neq 0$, we must also have $a(0) \neq 0$ and $b(0) \neq 0$, whence by (i) above, $a^*$, and $b^*$ are also non-constant. Since involution is self-inverse and $f(0) \neq0$, the other direction follows immediately. This establishes the claim.
	
	All that is left to show now, is that, given two irreducible polynomials $P$ and $P'$ such that $\deg P = \deg P'$, then $\Lambda_\rho(P^m) = \Lambda_\rho((P')^m)$, where $m$ is a positive integer. Recall that $P$ and $P'$ induce places $v = v_P$ and $w = v_{P'}$ respectively. By definition, we have
	\begin{align}
		\Lambda_\rho(P^m) &= (\deg P) \cdot \mathrm{Tr}(\rho_v(\Frob_v)^m \mid V_v),  \text{ and }  \\
		\Lambda_\rho((P')^m) &= (\deg P') \cdot \mathrm{Tr}(\rho_w(\Frob_w)^m \mid V_w).
	\end{align}
Therefore it suffices to show
\begin{equation} \label{eq:unimportant_1}
	\mathrm{Tr}(\rho_v(\Frob_v)^m \mid V_v) = \mathrm{Tr}(\rho_w(\Frob_w)^m \mid V_w),
\end{equation}
for all positive integers $m$. Since there exists only one finite field of a certain order up to isomorphism, we know 
\begin{equation}
	\Fqt/P \Fqt \cong \Fqt/P' \Fqt,
\end{equation}
and therefore in particular also $G_v \cong G_w$, and $\rho_v \cong \rho_w$. Thus $\rho_v(\Frob_v)$ and $\rho_w(\Frob_w)$ lie in the same conjugacy class. Clearly this implies \eqref{eq:unimportant_1}, and hence the lemma follows.
\end{proof}

The next lemma follows the proof given in \cite[Lemma 4.2]{keating2014variance}.

\begin{lemma}
	Let $B \in \Fqt$ be a polynomial of degree $\deg B = n-h-1$. Then 
	\begin{equation} \label{eq:fundamental_relation}
		\nu_\rho(n;t^{h+1}B,h) = \widetilde{\Psi}_\rho(n;B^*,t^{n-h}).
	\end{equation}
\end{lemma}
\begin{proof}
	Write $B(t) = b_{n-h-1}t^{n-h-1} + \dots + b_0$, and let $f(t) = a_dt^d + \dots + a_0$ be a polynomial such that $f(0) \neq 0$. By definition of  intervals, we have $f \in I(t^{h+1}B,h)$ if and only if $\deg(f(t)-t^{h+1}B(t)) \leq h$. Comparing coefficients gives that we must have $\deg(f) = n$, and
	\begin{equation}
		a_n = b_{n-h-1}, \dots, a_{n-h-1} = b_0.
	\end{equation}
	It is easy to see, that this is the same as requiring 
	\begin{equation}
		f^* \equiv B^* \; \mathrm{mod} \; t^{n-h}.
	\end{equation}
	Since $f(0) \neq 0$ it follows that $\deg(f^*)=n$, and so
	\begin{equation}
		\sum_{\substack{f \in I(t^{h+1}B,h) \\ f(0) \neq 0}}\Lambda_\rho(f) = \sum_{\substack{f^* \equiv B^* \; \mathrm{mod} \; t^{n-h} \\ \deg(f^*) = n}}\Lambda_\rho(f)
	\end{equation}
	Finally, as $\Lambda_\rho(f) = \Lambda_\rho(f^*)$ by the previous lemma, the result follows.
\end{proof}

\subsection{Mean and variance}
Denote by $\mathcal{P}_{\leq h}$ the set of polynomials of degree $\leq h$. Note that every $A \in \mathcal{M}_n$ can uniquely be written as 
\begin{equation}
	A = t^{h+1}B + g,
\end{equation}
where $B \in \mathcal{M}_{n-h-1}$ and $g \in \mathcal{P}_{\leq h}$. Since $\deg(g) \leq h$ we have $A \in I(t^{h+1}B,h)$. By the uniqueness of this decomposition, we may write the monic polynomials of degree $n$ as a disjoint union of intervals $I(t^{h+1}B,h)$, where $B$ runs through polynomials in $\mathcal{M}_{n-h-1}$. That is, we obtain
\begin{equation} \label{eq:disjoint_union}
	\mathcal{M}_n = \coprod_{B \in \mathcal{M}_{n-h-1}}I(t^{h+1}B,h)
\end{equation}
Note that if  $A, A' \in I(t^{h+1}B,h)$ then $I(A,h) = I(A',h) = I(t^{h+1}B,h)$ and so $\nu_\rho(A;h) = \nu_\rho(A';h)$. By using \eqref{eq:disjoint_union} and \eqref{eq:fundamental_relation} we get for the expectation
\begin{align}
	\mathbb{E}_{A,n}[\nu_\rho(A;h)] &= \frac{1}{q^n} \sum_{B \in \mathcal{M}_{n-h-1}} \sum_{A \in I(t^{h+1}B;h)} \nu_\rho(A;h) \\
	&= \frac{1}{\norm{\mathcal{M}_{n-h-1}}} \sum_{B \in \mathcal{M}_{n-h-1}}\nu_\rho(t^{h+1}B;h) \label{eq:nu_rho_mean_first_expression} \\
	&= \frac{1}{q^{n-h-1}} \sum_{\substack{B^* \; \mathrm{mod} \; t^{n-h} \\ B^*(0) = 1}} \widetilde{\Psi}_\rho(n;B^*,t^{n-h}),
\end{align}
and similarly for the variance we obtain 
\begin{align}
	\mathrm{Var}_{A,n}[\nu_\rho(A;h)]
	&= \frac{1}{q^{n-h-1}} \sum_{\substack{B^* \; \mathrm{mod} \; t^{n-h} \\ B^*(0) = 1}} \norm{\widetilde{\Psi}_\rho(n;B^*,t^{n-h}) - \mathbb{E}_{A,n}[\nu_\rho(A;h)]}^2. \label{eq:var_in_terms_of_arith}
\end{align}

\begin{lemma} \label{lem:expectation_short_intervals}
	Let $\rho \from G_{K,\mathcal{S}} \to \mathrm{GL}(V)$ be an $\ell$-adic Galois representation, then for polynomials $A \in \Fqt$ with $\deg A = n > h$ as above we have
	\begin{equation}
		\mathbb{E}_{A,n}[\nu_\rho(A;h)] = \frac{1}{q^{n-h-1}}\left( \sum_{f \in \mathcal{M}_n} \Lambda_\rho(f) - \Lambda_\rho(t^n) \right).
	\end{equation}
\end{lemma}
\begin{proof}
	Using the expression for the expectation that we obtained in \eqref{eq:nu_rho_mean_first_expression}, we compute
	\begin{equation}\label{eq:lemma_expression_1}
		\mathbb{E}_{A,n}[\nu_\rho(A;h)] = \frac{1}{q^{n-h-1}}  \sum_{B \in \mathcal{M}_{n-h-1}} \sum_{\substack{f \in I(t^{h+1}B,h) \\ f(0) \neq 0}} \Lambda_\rho(f).
	\end{equation}
	By definition of the interval around a polynomial, we can decompose $\mathcal{M}_n$ into intervals as in \eqref{eq:disjoint_union}, and so we get 
	\begin{equation} \label{eq:lemma_expression_3}
	 \coprod_{B \in \mathcal{M}_{n-h-1}}	\{ f \in I(t^{h+1}B,h) \colon f(0) \neq 0 \} = \{ f \in  \mathcal{M}_n \colon f(0) \neq 0 \}.
\end{equation}
The only monic polynomial $f \in \mathcal{M}_n$ with $f(0) = 0$, which is a prime power is given by $f(t) = t^n$. Since $\Lambda_\rho$ vanishes away from prime powers we get
\begin{equation} \label{eq:lemma_expression_2}
	\sum_{\substack{f \in \mathcal{M}_n \\ f(0) = 0}} \Lambda_\rho(f) = \Lambda_\rho(t^n).
\end{equation}
Combining \eqref{eq:lemma_expression_1}, \eqref{eq:lemma_expression_3} and \eqref{eq:lemma_expression_2} we therefore conclude the result.
\end{proof}
The next result will give us an expression of the variance of $\nu_\rho$ in terms of the cohomological traces. Before we state and prove the result, let us briefly recall the orthogonality relations for Dirichlet characters . For $f,g \in \Gamma(Q)$ we have
\begin{equation} \label{eq:char_orth_1}
	\frac{1}{\norm{\Phi(Q)}} \sum_{\varphi \in \Phi(Q)} \varphi(f) \overbar{\varphi(g)} = \begin{cases} 1 \quad &\text{if } f \equiv g \mod Q, \\
 	0 \quad &\text{otherwise},
 \end{cases}
\end{equation} 
and for $\varphi_1, \varphi_2 \in \Phi(Q)$ we get
\begin{equation} \label{eq:char_orth_2}
	\frac{1}{\norm{\Gamma(Q)}} \sum_{f \in \Gamma(Q)} \varphi_1(f) \overbar{\varphi_2(f)} = \begin{cases} 1 \quad &\text{if } \varphi_1 = \varphi_2, \\
		0 \quad &\text{otherwise}.
	\end{cases}
\end{equation}
From \eqref{eq:char_orth_2} it is not hard to deduce that when $Q = t^m$, for even characters $\varphi_1$ and $\varphi_2$, we have
\begin{equation} \label{eq:char_orth_3_even}
	\frac{1}{\norm{\Phi(Q)^{ev}}} \sum_{\substack{f \in  \Gamma(Q) \\ f(0) = 1}} \varphi_1(f) \overbar{\varphi_2(f)} = \begin{cases} 1 \quad &\text{if } \varphi_1 = \varphi_2, \\
		0 \quad &\text{otherwise}.
	\end{cases}
\end{equation}
For reference, a proof is given in \cite[Lemma 3.2]{keating2014variance}. In order to simplify notation slightly, we write $\Phi^*(Q)^{ev} = \Phi(Q)^{ev} \setminus \{ \varphi_{tr} \}$, where $\varphi_{tr}$ denotes the trivial Dirichlet character of modulus $Q$.
\begin{proposition}
	Let $\rho \from G_{K,\mathcal{S}} \to \mathrm{GL}(V)$ be a continuous Galois representation, let $Q = t^{n-h}$ and let $A \in \Fqt$ be polynomials such that $\deg A=n >h$. Then
	\begin{equation}
		\mathbb{E}_{A,n}[\nu_\rho(A;h)] = \frac{b_{\rho \otimes \varphi_{tr},n}}{q^{n-h-1}} \quad \text{and} \quad \mathrm{Var}_{A,n}[\nu_\rho(A;h)] = \frac{1}{q^{2(n-h-1)}} \sum_{\varphi \in \Phi^*(Q)^{ev}} \norm{b_{\rho \otimes \varphi,n}}^2.
	\end{equation}
	
\end{proposition}
\begin{proof}
	Given $f \in \Fqt$ we have for the trivial character of modulus $Q = t^{n-h}$
	\begin{equation}
		\varphi_{tr}(f) = \begin{cases} 1 \quad &\text{if } t \nmid f, \\ 0 \quad &\text{if } t \mid f. \end{cases}
	\end{equation}
	Thus by definition of the cohomological traces \eqref{eq:cohomological_traces_character_sum} we get
	\begin{equation}
		b_{\rho \otimes \varphi_{tr},n} = \sum_{f \in \mathcal{M}_n} \Lambda_\rho(f) \varphi_{tr}(f) = \sum_{\substack{f \in \mathcal{M}_n \\ f(0) \neq 0}} \Lambda_\rho(f).
	\end{equation}
	Comparing this with Lemma \ref{lem:expectation_short_intervals} gives the first part of the proposition.
	
For the second part, note that using the character orthogonality relation \eqref{eq:char_orth_1} we can rewrite $\tilde{\Psi}(n;B^*,Q)$ as
\begin{equation}
	\tilde{\Psi}(n;B^*,Q) = \frac{1}{\norm{\Phi(Q)}} \sum_{\varphi \in \Phi(Q)} \overbar{\varphi(B^*)} \sum_{\deg(f) = n} \Lambda_\rho(f) \varphi(f).
\end{equation}
	The term $\sum_{\deg(f) = n} \Lambda_\rho(f) \varphi(f)$ is non-zero only if $\varphi$ is an even character, and every even, non-trivial character contributes
	\begin{equation}
		\overbar{\varphi(B^*)} \frac{q-1}{\norm{\Phi(Q)}} \sum_{f \in \mathcal{M}_n} \Lambda_\rho(f) \varphi(f) = \frac{1}{q^{n-h-1}}\overbar{\varphi(B^*)} b_{\rho \otimes \varphi,n},
	\end{equation}
	where we used Lemma \ref{lem:even_characters} for the last equality. In the beginning of this proof we showed that the contribution of the trivial character is precisely $\mathbb{E}_{A,n}[\nu_\rho(A;h)]$, and therefore we obtain
	\begin{equation}
		\tilde{\Psi}(n;B^*,Q) - \mathbb{E}_{A,n}[\nu_\rho(A;h)] = \frac{1}{q^{n-h-1}}\sum_{\varphi \in \Phi^*(Q)^{ev}} \overbar{\varphi(B^*)} b_{\rho \otimes \varphi,n}.  
	\end{equation}
Comparing this with \eqref{eq:var_in_terms_of_arith} we get
\begin{equation}
	\mathrm{Var}_{A,n}\left[\nu_\rho(A;h)\right] = \frac{1}{q^{n-h-1}} \sum_{\substack{B^* \; \mathrm{mod} \; Q \\ B^*(0) = 1}} \frac{1}{q^{2(n-h-1)}} \norm{\sum_{\varphi \in \Phi^*(Q)^{ev}} \overbar{\varphi(B^*)} b_{\rho \otimes \varphi,n}}^2.
\end{equation} 
Exchanging the order of summation and using the character orthogonality relation \eqref{eq:char_orth_3_even} we therefore end up with
\begin{equation} \label{eq:variance_in_terms_b_rho_phi}
	\mathrm{Var}_{A,n}[\nu_\rho(A;h)]    = \frac{1}{q^{2(n-h-1)}} \sum_{\varphi \in \Phi^*(Q)^{ev}} \norm{b_{\rho \otimes \varphi,n}}^2,  
\end{equation}
which concludes the proof.
\end{proof}

\section{Weights, purity and characters} \label{section.weights_purity}
\subsection{Weights and purity of representations}
Recall that we fixed embeddings $\overbar{\Q} \hookrightarrow \overbar{\Q}_\ell$, and $\iota \from \overbar{\Q} \hookrightarrow \C$, even though we suppressed notation. 
The following definition, which we adopt from \cite{HKRG17}, will allow us to assume a Riemann hypothesis in a precise way.
\begin{definition}
	For a polynomial in $\overbar{\Q}_\ell[T]$, we say that it is \emph{$\iota$-pure of $q$-weight $w$}  if it is non-zero, all of its zeroes $\alpha$ lie in $\overbar{\Q}$ and they satisfy
	\begin{equation}
	\norm{\iota(\alpha)}^2 = \frac{1}{q^w}.	
	\end{equation}
	Further, we say that a polynomial is \emph{$\iota$-mixed of $q$-weights $\leq w$}  if it is a product of polynomials, each of which is $\iota$-pure of $q$-weight $\leq w$.
	
	For an $\ell$-adic Galois representation $\rho \from G_{K,\mathcal{S}} \to \mathrm{GL}(V)$ if there exists a finite subset of places $\mathcal{S} \subset \mathcal{P}$ such that $\rho$ is unramified away from $\mathcal{S}$, we say that $\rho$ is \emph{pointwise $\iota$-pure of $q$-weight $w$} if for each $v \notin \mathcal{S}$, the Euler factor $L(T^{d_v},\rho_v)$ is $\iota$-pure of $q$-weight $w$.
\end{definition}

We would like to have a Riemann hypothesis satisfied not just for the $L$-function of $\rho$, but in particular for all twists $\rho \otimes \varphi$. The following result, which is shown in \cite{HKRG17}, ensures that we do not need too many assumptions.
\begin{lemma}
	let $\rho$ be a Galois representation as before, and let $\varphi \in \Phi(Q)$ be a Dirichlet character. If $\rho$ is pointwise $\iota$-pure of $q$-weight $w$, then so is $\rho \otimes \varphi$.
\end{lemma}
\begin{proof}
	As before let $\mathcal{Q}$ be the set of places, which divide $Q$. We wish to show that for $v \notin \mathcal{Q}$, the Euler factors $L(T^{d_v},(\rho \otimes \varphi)_v)$ are $\iota$-pure of $q$-weight $w$. For this note that since $\Gamma(Q)$ has finite order, we must have that $\varphi_\mathcal{Q}(\Frob_{v_P}) = \varphi(P) = \zeta$ is a root of unity, whence also $\zeta \in \overbar{\Q}$. Comparing this with \eqref{eq:euler_factor_twist_nice}, we see immediately that $\alpha$ is a zero of $L(T,(\rho \otimes \varphi)_v)$  if and only if $\alpha/\zeta$ is a zero of $L(T,\rho_v)$. The result follows.
\end{proof}
Deligne famously proved the Riemann hypothesis for varieties over finite fields in the most general fashion in 1980. 
Theorem 1 and Theorem 2 of \cite{deligne1980conjecture} imply the following result.
\begin{theorem}[Deligne]
	Let $\rho$ be a Galois representation as above, and let $\varphi \in \Phi(Q)$ be a Dirichlet character. Assume $\rho \otimes \varphi$ is pointwise $\iota$-pure of weight $w$. Then
	\begin{itemize}
		\item $L_\mathcal{Q}(T,\rho \otimes \varphi)$ is a ratio of polynomials, which are $\iota$-mixed of $q$-weights $\leq w+2$, and
		\item $L(T,\rho \otimes \varphi)$ is a ratio of polynomials in $\overbar{\Q}[T]$, which are  $\iota$-pure of $q$-weight $w+1$.
	\end{itemize}
\end{theorem}
\subsection{Good, mixed and heavy characters}
From now on we will always assume that $\rho$ is pointwise $\iota$-pure of $q$-weight $w$ and $Q = t^{n-h}$, where $n-h \geq 5$. We would like to consider characters $\varphi$ for which $L_\mathcal{Q}(T,\rho \otimes \varphi)$ is a pure polynomial, so that we obtain a good estimate of $b_{\rho \otimes \varphi,n}$. 

Similarly to Hall, Keating and Roditty-Gershon in \cite{HKRG17} we will distinguish between certain families of characters.
\begin{definition}
	Let $\rho$ be pointwise $\iota$-pure of $q$-weight $w$. Let $Q \in \Fqt$ be a polynomial. For an even character $\varphi \in \Phi(Q)$ we say that $\varphi$ is \emph{good}  if 
	\begin{equation}
		M(T,\rho \otimes \varphi) \coloneqq L_\mathcal{Q}(T,\rho \otimes \varphi)/(1-T)
	\end{equation} 
	 is a polynomial that is $\iota$-pure of $q$-weight $w+1$. We denote the degree of $M$ by $s_\mathcal{Q}(\rho)$ or if the context is clear, just by $S$.
	 If an even character is not good, we call it \emph{bad}. We call a bad character  \emph{heavy}  if $L(T,\rho \otimes \varphi)$ is not a polynomial, and we call a bad character \emph{mixed}  if it is not heavy.
	 The sets of good, heavy and mixed even characters are respectively denoted by
	 \begin{equation}
	 	\Phi(Q)_{\rho \; \mathrm{good}}^{ev}, \quad \Phi(Q)_{\rho \; \mathrm{heavy}}^{ev} \quad \text{and} \quad \Phi(Q)_{\rho \; \mathrm{mixed}}^{ev}.
	 \end{equation}
\end{definition}
\begin{remark}
	 A priori the degree $S$ is not well-defined since it might vary across different Dirichlet characters $\varphi$. Fortunately, for our setting \cite[Theorem 1.2]{sawin2018equidistribution} provides that there is a single $S$ for $q^{n-h-1}(1-O(1/q))$ even characters. In fact, a precise value is given in \cite[Lemma 4.10]{sawin2018equidistribution} provided certain conditions are satisfied.
\end{remark}
\subsection{Cohomological interpretation}
Note that our Galois representation $\rho$ is finitely ramified and thus it can be regarded as a representation of the \'etale fundamental group of some open subset of $\mathbb{A}^1_{\overbar{\mathbb{F}}_{q}}$. This defines a middle extension sheaf $\mathrm{ME}(\rho)$ on $\mathbb{A}^1_{\overbar{\mathbb{F}}_{q}}$ that is lisse on a dense open subset. Assuming that $V$ is irreducible implies that  $\mathrm{ME}(\rho)$ is irreducible too. For a more detailed discussion of  $\mathrm{ME}(\rho)$ we refer the reader to \cite{HKRG17}.

For a good or mixed even Dirichlet character $\varphi \in \Phi(Q)$ write $\mathcal{L}_\varphi$ for the unique rank one lisse sheaf whose monodromy representation is $\varphi$. Further write $C_{\varphi}$ for the scaled conjugacy class of $\Frob_q/q^{\frac{w+1}{2}}$ on $H_c^1(\mathbb{A}^1_{\overbar{\mathbb{F}}_{q}}, \mathrm{ME}(\rho) \otimes \mathcal{L}_\varphi)$, where $H_c^i(\mathbb{A}^1_{\overbar{\mathbb{F}}_{q}},  \mathrm{ME}(\rho) \otimes \mathcal{L}_\varphi)$ is the $i$-th \'etale cohomology group of the sheaf $\mathrm{ME}(\rho) \otimes \mathcal{L}_\varphi$.

Characteristic polynomials remain unchanged after semisimplification and therefore $\det(1-C_\varphi q^{\frac{w+1}{2}} T ) = \det (1-T \; \Frob_q, H_c^1(\mathbb{A}^1_{\overbar{\mathbb{F}}_{q}}, \mathrm{ME}(\rho) \otimes \mathcal{L}_\varphi))$. 
Recalling the general formula for an Artin $L$-function of a Galois representation in terms of the Frobenius action on the first cohomology of the associated middle extension sheaf we find
\begin{equation} \label{eq:twisted_l_cohomological}
	L_{\mathcal{Q}}(T, \rho \otimes \varphi)/(1-T) = \det \left( 1-T\; \Frob_q, H_c^1(\mathbb{A}^1_{\overbar{\mathbb{F}}_{q}}, \mathrm{ME}(\rho) \otimes \mathcal{L}_\varphi) \right).
\end{equation}
The factor of $1/(1-T)$ accounts for the factor at the place at zero since we consider the completed $L$-function.

\subsection{Good characters dominate}
For a character $\varphi \in \Phi(Q)^{ev}$ we say that the space $H_c^1(\mathbb{A}^1_{\overbar{\mathbb{F}}_{q}}, \mathrm{ME}(\rho) \otimes \mathcal{L}_\varphi)$ is $\iota$-pure of $q$-weight $w$ if  $\det \left( 1-T\; \Frob_q, H_c^1(\mathbb{A}^1_{\overbar{\mathbb{F}}_{q}}, \mathrm{ME}(\rho) \otimes \mathcal{L}_\varphi) \right)$ is $\iota$-pure of $q$-weight $w$. We restate \cite[Lemma 2.7]{sawin2018equidistribution}.
\begin{lemma} 
	Let $\rho$ be pointwise $\iota$-pure with $q$-weight $w$. Then the space $H_c^1(\mathbb{A}^1_{\overbar{\mathbb{F}}_{q}}, \mathrm{ME}(\rho) \otimes \mathcal{L}_\varphi)$ is $\iota$-mixed of $q$-weight $\leq w+1$, and in particular it is $\iota$-pure of $q$-weight $w+1$ for all but $\dim V$ characters in $\Phi(Q)$.
\end{lemma}
Using \eqref{eq:twisted_l_cohomological} we therefore deduce the following.
\begin{lemma} \label{lem.mixed_chars_finite}
	Let $\varphi \in \Phi(Q)_{\rho \; \mathrm{mixed}}^{ev}$. Then $L_{\mathcal{Q}}(T, \rho \otimes \varphi)/(1-T)$ is mixed of $q$-weights $\leq w+1$. Further the number of mixed characters is finite, in particular $\norm{\Phi(Q)_{\rho \; \mathrm{mixed}}^{ev}} = O(1)$ as $q \rightarrow \infty$.
\end{lemma}
Therefore we are now only left with dealing with heavy characters. In the final theorem we will assume that the only possible bad character is $\varphi_{tr}$. The following shows that this is usually not a very strong condition.
\begin{lemma}[Corollary 8.3.3 in \cite{HKRG17}]
	Assume that $\rho$ is pointwise $\iota$-pure with $q$-weight $w$. Assume further that $\rho$ is irreducible and geometrically semisimple. Write $m = \dim V$. Then $\Phi(Q)_{\rho \; \mathrm{heavy}}^{ev} \subseteq \{ \varphi_{tr} \}$ if and only if one of the following hold
	\begin{itemize}
		\item[(i)] $m > 1$.
		\item[(ii)] $m = 1$ and $\rho$ is geometrically isomorphic to the trivial representation.
		\item[(iii)] $m=1$ and $\rho$ is not geometrically isomorphic to a Dirichlet character in $\Phi(Q)$.
	\end{itemize}
	In particular equality holds if and only if (ii) holds.
\end{lemma}

Assume that $\varphi \in \Phi(Q)$ is  good character. Then we can write 
\begin{equation}
	L_\mathcal{Q}(T,\rho \otimes \varphi) = (1-T) \prod_{i=1}^S (1-\alpha_i T),
\end{equation}
where $\norm{\alpha_i} = q^{(1+w)/2}$. Therefore, there exists a conjugacy class of unitary matrices, that depends only on $\rho$ and $\varphi$, of which we will denote a representative by $\theta_{\rho,\varphi} \in \U(S)$, such that
\begin{equation}
	L_{\mathcal{Q}}(T,\rho \otimes \varphi) = (1-T) \det(I - q^{\frac{1+w}{2}} \theta_{\rho,\varphi}T).
\end{equation}
Recall from \eqref{eq:twisted_L-function_in_terms_of_coh_traces} that we have
\begin{equation}
	T \frac{d}{dT} \log L_\mathcal{Q}(T,\rho \otimes \varphi	) = \sum_{n=1}^\infty b_{\rho \otimes \varphi,n} T^n,
\end{equation}
so that for a good character $\varphi$ we obtain
\begin{equation}
	b_{\rho \otimes \varphi,n} = - q^{\frac{n(1+w)}{2}}\mathrm{Tr}(\theta_{\rho,\varphi}^n) - 1.
\end{equation}
Thus, using \eqref{eq:variance_in_terms_b_rho_phi}, we obtain the following important identity
\begin{multline} \label{eq:variance_almost_there}
	\frac{1}{q^{nw+h+1}} \mathrm{Var}_{A,n} \left[ \nu_\rho(A;h) \right] = \frac{1}{q^{n-h-1}}\sum_{\varphi \in \Phi^*(Q)^{ev}_{\rho \; \mathrm{good}}} \norm{\mathrm{Tr}(\theta_{\rho,\varphi}^n)}^2 \\
	+ \frac{1}{q^{n(w+2)-h-1}} \sum_{\varphi \in \Phi^*(Q)^{ev}_{\rho \; \mathrm{bad}}} \norm{b_{\rho \otimes \varphi,n}}^2 + O\left(q^{h+1-n-\frac{n(1+w)}{2}}\right).
\end{multline}
If we further assume that $\rho$ is $\iota$-pure of $q$-weight $w$ then from Lemma \ref{lem.mixed_chars_finite} we deduce
\begin{equation} \label{eq:variance_good_chars_dominate}
	\frac{1}{q^{nw+h+1}} \mathrm{Var}_{A,n} \left[ \nu_\rho(A;h) \right] = \frac{1}{q^{n-h-1}}\sum_{\varphi \in \Phi^*(Q)^{ev}_{\rho \; \mathrm{good}}} \norm{\mathrm{Tr}(\theta_{\rho,\varphi}^n)}^2 + O\left( \frac{1}{q} \right)
\end{equation}

\section{Equidistribution \& proof of the main theorem}
\begin{definition}
	Let $f \from \U(S) \to \C$ be a  continuous and conjugacy-invariant function. We define the \emph{mean}  of $f$ to be the unique continuous function $\langle f \rangle \from \U(1) \to \C$ satisfying
	\begin{equation}
		\int_{\mathrm{U}(S)} f(g) \psi(\det g) dg = \int_{\mathrm{U}(S)} \langle f \rangle (\det g) \psi(\det g) dg,
	\end{equation}
	for any continuous $\psi \from \U(1) \to \C$.
\end{definition}
As a special case adjusted to our setting, Theorem~1.2 of \cite{sawin2018equidistribution} can be stated as follows.
\begin{theorem}[Sawin] \label{thm:equidistribution}
	Let $\rho$ be a Galois representation, which is pointwise $\iota$-pure of $q$-weight $w$,  let $f \from \U(S) \to \U(1)$ be a continuous, conjugacy-invariant function and let $Q = t^{n-h}$. For $\varphi \in \Phi^*(Q)_{\rho \; \mathrm{good}}^{ev}$ we have representatives $\theta_{\rho,\varphi} \in \U(S)$ of conjugacy classes such that $L_\mathcal{Q}(T,\rho \otimes \varphi) = (1-T) \det(I-q^{\frac{1+w}{2}}\theta_{\rho,\varphi}T)$. 
	
	If $n-h \geq 5$, then
	\begin{equation}
		\lim_{q \to \infty} \left( \frac{\sum_{\varphi \in \Phi^*(Q)^{ev}_{\rho \; \mathrm{good}}} f(\theta_{\rho,\varphi})}{\norm{\Phi^*(Q)_{\rho \; \mathrm{good}}^{ev}}}   - \frac{\sum_{\varphi \in \Phi^*(Q)^{ev}_{\rho \; \mathrm{good}}} \langle f \rangle (\det(\theta_{\rho,\varphi})) }{\norm{\Phi^*(Q)_{\rho \; \mathrm{good}}^{ev}}} \right) = 0.
	\end{equation}
\end{theorem}
In our case we are dealing with $f(g) = \norm{\mathrm{Tr}(g^n)}^2$. The proof of the next Lemma is based on \cite[Lemma 3.3]{gorodetsky2018variance}.
\begin{lemma} \label{lem:mean_of_traces}
	Let $f \from \U(S) \to \C$ be given by $f(g) = \norm{\mathrm{Tr}(g^n)}^2$. Then
	\begin{equation}
		\langle f \rangle (z) = \int_{g \in \U(S)} f(g) dg = \mathrm{min}\{ n,S \},
\end{equation}
for any $z \in \U(1)$.
\end{lemma}
\begin{proof}
	By definition of $\langle f \rangle$ we need to show 
	\begin{equation}
		\int_{\U(S)} f(g) \psi(\det g) dg = \left(\int_{\U(S)} f(g) dg \right) \left( \int_{\U(S)}  \psi(\det g) dg \right),
	\end{equation}
	for any continuous function $\psi \from \U(1) \to \C$. By the complex Stone-Weierstrass theorem, polynomials are dense in the space of continuous functions on the compact set $\U(1) \subset \C$, so that it is enough to show the above equality for all $\psi(z) = z^k$, $k \in \mathbb{Z}$. Since the Haar measure has total unit mass, the case $k=0$ is trivial. 
	
	Note that $f(g) = \norm{\mathrm{Tr}(g^n)}^2 = \mathrm{Tr}(g^n) \overbar{\mathrm{Tr}(g^n)}$ is a Laurent polynomial in the eigenvalues of $g$ of degree $0$. This implies, that for any $\lambda \in \U(1)$ we have $f(\lambda g) = f(g)$. If $k \neq 0$ by translation invariance of the Haar measure we have for any $\lambda \in \U(1)$,
	\begin{equation}
		\int_{ \U(S)} f(g) (\det g)^k dg = \int_{\U(S)} f(\lambda g) (\det \lambda g)^k dg = \lambda^{kR} \int_{ \U(S)} f(g) (\det g)^k dg,
	\end{equation}
	so that considering $\lambda \neq 1$ implies
	\begin{equation}
		\int_{ \U(S)} f(g) (\det g)^k dg = \left(\int_{\U(S)} f(g) dg \right) \left( \int_{\U(S)}  (\det g)^k dg \right) = 0.
	\end{equation}
The last part of the proof follows, since a standard matrix integral evaluation shows
\begin{equation}
	\int_{\U(S)} \norm{\mathrm{Tr}(g^n)}^2 dg = \mathrm{min}\{ n,S \}.
\end{equation}
A proof of this fact can be found in \cite[Theorem 2.1]{diaconis2001linear}.
\end{proof}
We are now finally in a position to state and prove the main result of this chapter.
\begin{theorem} \label{thm:main_result_full_form}
	Let $\rho$ be a Galois representation, which is pointwise $\iota$-pure of $q$-weight $w$. Let $A$ be polynomials of degree $n$, and let $h$ be a positive integer such that $h \leq n-5$. For $Q = t^{n-h}$, assume further that $\Phi(Q)^{ev}_{\rho \; \mathrm{heavy}} \subseteq \{ \varphi_{tr} \}$. Then
	\begin{equation}
		\lim_{q \to \infty} \frac{1}{q^{nw+h+1}} \mathrm{Var}_{A,n} \left[ \nu_\rho(A;h) \right] = \mathrm{min}\{n,s_\mathcal{Q}(\rho) \}.
	\end{equation}
\end{theorem}
\begin{proof}
	Write $Q=t^{n-h}$. From Lemma \ref{lem.mixed_chars_finite} and the assumption $\Phi(Q)^{ev}_{\rho \; \mathrm{heavy}} \subseteq \{ \varphi_{tr} \}$ we find
	\begin{equation}
		q^{n-h-1} = \norm{\Phi(Q)^{ev}} \sim \norm{\Phi^*(Q)^{ev}_{\rho \; \mathrm{good}}}.
	\end{equation} 
	Using this and \eqref{eq:variance_good_chars_dominate} thus gives
	\begin{equation}
		\lim_{q \to \infty}\frac{1}{q^{nw+h+1}} \mathrm{Var}_{A,n} \left[ \nu_\rho(A;h) \right] = \lim_{q \to \infty} \frac{1}{\norm{\Phi^*(Q)^{ev}_{\rho \; \mathrm{good}}}}\sum_{\varphi \in \Phi^*(Q)^{ev}_{\rho \; \mathrm{good}}} \norm{\mathrm{Tr}(\theta_{\rho,\varphi}^n)}^2.
	\end{equation}
	Note that all conditions of Theorem~\ref{thm:equidistribution} are satisfied. Thus, by the result of Lemma \ref{lem:mean_of_traces}, applying Theorem \ref{thm:equidistribution} we obtain
	\begin{equation}
		\lim_{q \to \infty} \frac{1}{\norm{\Phi^*(Q)^{ev}_{\rho \; \mathrm{good}}}}\sum_{\varphi \in \Phi^*(Q)^{ev}_{\rho \; \mathrm{good}}} \norm{\mathrm{Tr}(\theta_{\rho,\varphi}^n)}^2  = \mathrm{min}\{ n,s_\mathcal{Q}(\rho) \},
	\end{equation}
	which completes the proof.
\end{proof}

\bibliographystyle{alpha}
\bibliography{variance_gen_von_Mangoldt.bib}

\end{document}